\documentclass[a4paper]{amsart}

\title[Bilateral Gamma Distributions and Processes in Finance]{Bilateral Gamma Distributions and Processes in Financial Mathematics}
\author{Uwe K{\"u}chler \and Stefan Tappe}
\thanks{We are grateful to Michael S\o{}rensen and an anonymous referee for their helpful remarks and discussions.}

\usepackage{amscd}
\usepackage{amsmath}
\usepackage{amssymb}
\usepackage{amsthm}
\usepackage{bbm}
\usepackage{eucal}

\newif\ifpdf
\ifx\pdfoutput\undefined
   \pdffalse        % we are not running PDFLaTeX
\else
   \pdfoutput=1     % we are running PDFLaTeX
   \pdftrue
\fi

\ifpdf
   \usepackage[pdftex]{graphicx}
\else
   \usepackage{graphicx}
\fi

\numberwithin{equation}{section}
\swapnumbers

\newtheorem{satz}{Satz}[section]

\newtheorem{theorem}[satz]{Theorem}
\newtheorem{proposition}[satz]{Proposition}
\newtheorem{corollary}[satz]{Corollary}
\newtheorem{lemma}[satz]{Lemma}

\newcommand{\abs}[1]{\lvert #1 \rvert}
\renewcommand{\Re}{\operatorname{Re}}

\begin{document}

\maketitle\thispagestyle{empty}

\begin{abstract}
We present a class of L\'evy processes for modelling financial market fluctuations: Bilateral Gamma processes. Our starting point is to explore the properties of bilateral Gamma distributions, and then we turn to their associated L\'evy processes. We treat exponential L\'evy stock models with an underlying bilateral Gamma process as well as term structure models driven by
bilateral Gamma processes and apply our results to a set of real financial data (DAX 1996-1998). \bigskip

\end{abstract}


\section{Introduction}

In recent years more realistic stochastic models for price movements in financial markets
have been developed, for example by replacing the classical Brownian motion by L{\'e}vy processes.
Popular examples of such L\'evy processes are generalized hyperbolic processes \cite{Barndorff-Nielsen} and their subclasses, Variance Gamma processes \cite{Madan} and CGMY-processes \cite{CGMY}. A survey about L\'evy processes used for applications to finance can for instance be found in \cite[Chap. 5.3]{Schoutens}. 

We propose another family of L{\'e}vy processes which seems to be interesting: Bilateral Gamma processes, which are defined as the difference of two independent Gamma processes. This four-parameter class of processes is more flexible than Variance Gamma processes, but still analytically tractable, in particular these processes have a simple cumulant generating function.

The aim of this article is twofold: First, we investigate the properties of these processes
as well as their generating distributions, and show how they are related to other distributions
considered in the literature. 

As we shall see, they have a series of properties making
them interesting for applications: Bilateral Gamma distributions are selfdecomposable, stable under convolution and have a simple cumulant generating function. The associated L\'evy processes are finite-variation processes making infinitely many jumps at each interval with positive length, and all their increments are bilateral Gamma distributed. In particular, one can easily provide simulations for the trajectories of bilateral Gamma processes.

So, our second goal
is to apply bilateral Gamma processes for modelling financial market fluctuations.
We treat exponential L{\'e}vy stock market models and derive a closed formula for pricing European Call Options. As an illustration, we apply our results to the evolution of the German stock index DAX over the period of three years. Term structure models driven by
bilateral Gamma processes are considered as well.

\section{Bilateral Gamma distributions}\label{sec-distribution}

A popular method for building L{\'e}vy processes is to take a subordinator
$S$, a Brownian motion $W$ which is independent of $S$, and to construct
the time-changed Brownian motion $X_t := W(S_t)$. For instance, generalized hyperbolic processes and Variance Gamma processes are constructed in this fashion. We do not go this way. Instead,
we define $X := Y - Z$ as the difference of two independent subordinators $Y,Z$.
These subordinators should have a simple characteristic function, because then
the characteristic function of the resulting L{\'e}vy process $X$ will be simple, too.
Guided by these ideas, we choose Gamma processes as subordinators. %\bigskip

To begin with, we need the following slight generalization of
Gamma distributions. For $\alpha > 0$ and $\lambda \in \mathbb{R}
\setminus \{ 0 \}$, we define the $\Gamma(\alpha,\lambda)$-distribution by the density 
\begin{align*}
f(x) = \frac{|\lambda|^{\alpha}}{\Gamma(\alpha)} |x|^{\alpha - 1}
e^{-|\lambda| |x|} \left( \mathbbm{1}_{\{ \lambda > 0\}} \mathbbm{1}_{\{ x > 0 \}}
+ \mathbbm{1}_{\{ \lambda < 0\}} \mathbbm{1}_{\{ x < 0 \}} \right), \quad x \in \mathbb{R}.
\end{align*}
If $\lambda > 0$, then this is just the well-known Gamma distribution, and for $\lambda < 0$ one has a Gamma distribution concentrated on the negative half axis.
One verifies that for each $(\alpha,\lambda) \in (0,\infty) \times \mathbb{R} \setminus \{ 0 \}$
the characteristic function of a $\Gamma(\alpha,\lambda)$-distribution is given by
\begin{align}\label{cf-gamma}
\varphi(z) = \left( \frac{\lambda}{\lambda -
iz} \right)^{\alpha}, \quad z \in \mathbb{R}
\end{align}
where the power $\alpha$ stems from the main branch of the complex logarithm. %\bigskip

A {\em bilateral Gamma distribution} with parameters
$\alpha^+, \lambda^+ ,\alpha^-, \lambda^- > 0$ is defined as the convolution
\begin{align*}
\Gamma(\alpha^+,\lambda^+;\alpha^-,\lambda^-) :=
\Gamma(\alpha^+,\lambda^+) * \Gamma(\alpha^-,-\lambda^-).
\end{align*}
Note that for independent random variables $X,Y$ with $X \sim \Gamma(\alpha^+,\lambda^+)$ and $Y \sim \Gamma(\alpha^-,\lambda^-)$ the difference has a bilateral Gamma distribution $X - Y \sim \Gamma(\alpha^+,\lambda^+;\alpha^-,\lambda^-)$. 

By (\ref{cf-gamma}), the characteristic function of a bilateral Gamma distribution is
\begin{align}\label{cf-bilateral}
\varphi(z) = \left(
\frac{\lambda^+}{\lambda^+ - iz} \right)^{\alpha^+} \left(
\frac{\lambda^-}{\lambda^- + iz} \right)^{\alpha^-}, \quad z \in \mathbb{R}.
\end{align}

\begin{lemma}\label{sums-of-bilateral} \mbox{}
\begin{enumerate}
\item Suppose $X \sim \Gamma(\alpha_1^+,\lambda^+;\alpha_1^-,\lambda^-)$ and $Y \sim
\Gamma(\alpha_2^+,\lambda^+;\alpha_2^-,\lambda^-)$,
and that $X$ and $Y$ are independent. Then $X + Y \sim \Gamma(\alpha_1^+ + \alpha_2^+,\lambda^+;\alpha_1^- +
\alpha_2^-,\lambda^-)$.

\item For $X \sim \Gamma(\alpha^+,\lambda^+;\alpha^-,\lambda^-)$ and $c > 0$ it holds
$cX \sim \Gamma(\alpha^+,\frac{\lambda^+}{c};\alpha^-,\frac{\lambda^-}{c})$.

\end{enumerate}
\end{lemma}

\begin{proof}
The asserted properties follow from expression (\ref{cf-bilateral}) of the characteristic function.
\end{proof}

As it is seen from the characteristic function (\ref{cf-bilateral}),
bilateral Gamma distributions are stable under convolution, and they are \textit{infinitely divisible}. It follows from \cite[Ex. 8.10]{Sato} that both, the drift and the Gaussian part in the L\'evy-Khintchine formula (with truncation function $h = 0$), are equal to zero, and that the L\'evy measure is given by
\begin{align}\label{levy-measure}
F(dx) = \left( \frac{\alpha^+}{x} e^{-\lambda^+ x}
\mathbbm{1}_{(0,\infty)}(x) + \frac{\alpha^-}{|x|} e^{-\lambda^-
|x|} \mathbbm{1}_{(-\infty,0)}(x) \right)dx.
\end{align}
Thus, we can also express the characteristic function $\varphi$ as
\begin{align}\label{char-fkt-self}
\varphi(z) = \exp \left( \int_{\mathbb{R}} \left( e^{izx} - 1 \right) \frac{k(x)}{x} dx \right), \quad z \in \mathbb{R}
\end{align}
where $k : \mathbb{R} \rightarrow \mathbb{R}$ is the function
\begin{align}\label{fkt-k-self}
k(x) = \alpha^+ e^{-\lambda^+ x} \mathbbm{1}_{(0,\infty)}(x) - \alpha^- e^{-\lambda^- |x|} \mathbbm{1}_{(-\infty,0)}(x), \quad x \in \mathbb{R}
\end{align}
which is decreasing on each of $(-\infty,0)$ and $(0,\infty)$.
It is an immediate consequence of \cite[Cor. 15.11]{Sato} that bilateral Gamma distributions are \textit{selfdecomposable}. By (\ref{levy-measure}), it moreover holds
\begin{align*}
\int_{|x| > 1} e^{zx} F(dx) < \infty \quad \text{for all $z \in (-\lambda^-,\lambda^+)$.}
\end{align*}
Consequently, the \textit{cumulant generating function}
\begin{align*}
\Psi(z) = \ln \mathbb{E} \left[ e^{zX} \right] \quad \text{(where $X \sim \Gamma(\alpha^+,\lambda^+;\alpha^-,\lambda^-)$)}
\end{align*}
exists on $(-\lambda^-,\lambda^+)$, and $\Psi$ and $\Psi'$ are, with regard to (\ref{cf-bilateral}), given by
\begin{align}\label{cumulant-gamma}
\Psi(z) &= \alpha^+ \ln \left( \frac{\lambda^+}{\lambda^+ - z}
\right) + \alpha^- \ln \left( \frac{\lambda^-}{\lambda^- + z} \right),
\quad z \in (-\lambda^-, \lambda^+),
\\ \label{cumulant-gamma-derivative} \Psi'(z) &= \frac{\alpha^+}{\lambda^+ - z} - \frac{\alpha^-}{\lambda^- + z}, \quad z \in (-\lambda^-, \lambda^+).
\end{align}
Hence, the $n$-th order cumulant $\kappa_n = \frac{\partial^n}{\partial
z^n} \Psi(z) |_{z=0}$ is given by
\begin{align}\label{cumulant-kappa-gamma}
\kappa_n = (n-1)! \left( \frac{\alpha^+}{(\lambda^+)^n} + (-1)^n
\frac{\alpha^-}{(\lambda^-)^n} \right), \quad n \in \mathbb{N} = \{ 1,2,\ldots \}.
\end{align}
In particular, for a
$\Gamma(\alpha^+,\lambda^+;\alpha^-,\lambda^-)$-distributed random
variable $X$, we can specify
\begin{itemize}
\item The expectation
\begin{align}\label{expectation-bilateral}
\mathbb{E}[X] = \kappa_1 = \frac{\alpha^+}{\lambda^+} -
\frac{\alpha^-}{\lambda^-}.
\end{align}

\item The variance
\begin{align}\label{variance-bilateral}
{\rm Var}[X] = \kappa_2 = \frac{\alpha^+}{(\lambda^+)^2} +
\frac{\alpha^-}{(\lambda^-)^2}.
\end{align}

\item The Charliers skewness
\begin{align}
\gamma_1(X) = \frac{\kappa_3}{\kappa_2^{3/2}} = \frac{2 \left( \frac{\alpha^+}{(\lambda^+)^3} - \frac{\alpha^-}{(\lambda^-)^3} \right)}{\left(
\frac{\alpha^+}{(\lambda^+)^2} + \frac{\alpha^-}{(\lambda^-)^2}
\right)^{3/2}}.
\end{align}

\item The kurtosis
\begin{align}
\gamma_2(X) = 3 + \frac{\kappa_4}{\kappa_2^2} = 3 + \frac{6 \left( \frac{\alpha^+}{(\lambda^+)^4} +
\frac{\alpha^-}{(\lambda^-)^4} \right)}{\left(
\frac{\alpha^+}{(\lambda^+)^2} + \frac{\alpha^-}{(\lambda^-)^2}
\right)^2}.
\end{align}
\end{itemize}

It follows that bilateral Gamma distributions are \textit{leptokurtic}.

\section{Related classes of distributions}\label{sec-related}

As apparent from the L{\'e}vy measure (\ref{levy-measure}), bilateral
Gamma distributions are special cases of \textit{generalized tempered
stable distributions} \cite[Chap. 4.5]{Cont-Tankov}. This six-parameter family is
defined by its L{\'e}vy measure
\begin{align*}
F(dx) = \left( \frac{\alpha^+}{x^{1+\beta^+}} e^{-\lambda^+ x}
\mathbbm{1}_{(0,\infty)}(x) + \frac{\alpha^-}{|x|^{1+\beta^-}} e^{-\lambda^-
|x|} \mathbbm{1}_{(-\infty,0)}(x) \right)dx.
\end{align*}
The \textit{CGMY-distributions}, see \cite{CGMY}, are a four-parameter family with L{\'e}vy measure
\begin{align*}
F(dx) = \left( \frac{C}{x^{1+Y}} e^{-M x}
\mathbbm{1}_{(0,\infty)}(x) + \frac{C}{|x|^{1+Y}} e^{-G
|x|} \mathbbm{1}_{(-\infty,0)}(x) \right)dx.
\end{align*}
We observe that
some bilateral Gamma distributions are
CGMY-distributions, and vice versa. %\bigskip

As the upcoming result reveals, bilateral Gamma distributions are not closed under weak
convergence.

\begin{proposition}
Let $\lambda^+, \lambda^- > 0$ be arbitrary. Then the following
convergence holds:
\begin{align*}
\Gamma \left( \frac{(\lambda^+)^2 \lambda^- n}{\lambda^+ +
\lambda^-} ,\lambda^+ \sqrt{n}; \frac{\lambda^+ (\lambda^-)^2
n}{\lambda^+ + \lambda^-}, \lambda^- \sqrt{n} \right)
\overset{w}{\rightarrow} N(0,1) \quad \text{for $n \rightarrow
\infty$.}
\end{align*}
\end{proposition}

\begin{proof}
This is a consequence of the Central
Limit Theorem, Lemma \ref{sums-of-bilateral} and relations
(\ref{expectation-bilateral}), (\ref{variance-bilateral}).
\end{proof}

Bilateral Gamma distributions are special cases of \textit{extended
generalized Gamma convolutions} in the terminology of \cite{Thorin}. These are all infinitely divisible distributions $\mu$ whose characteristic function is of the form
\begin{align*}
\hat{\mu}(z) = \exp \left( izb - \frac{cz^2}{2} - \int_{\mathbb{R}} \left[ \ln \left( 1 - \frac{iz}{y} \right)
+ \frac{izy}{1 + y^2} \right] dU(y) \right), \quad z \in \mathbb{R}
\end{align*}
with $b \in \mathbb{R}, c \geq 0$ and a non-decreasing function $U
: \mathbb{R} \rightarrow \mathbb{R}$ with $U(0) = 0$ satisfying
the integrability conditions
\begin{align*}
\int_{-1}^1 |\ln y|dU(y) < \infty \quad \text{and} \quad \int_{-\infty}^{-1}
\frac{1}{y^2} dU(y) + \int_1^{\infty} \frac{1}{y^2} dU(y) <
\infty.
\end{align*}

Since extended generalized Gamma convolutions are closed under
weak limits, see \cite{Thorin}, every limiting case of bilateral
Gamma distributions is an extended generalized Gamma convolution. %\bigskip

Let $Z$ be a subordinator (an increasing real-valued L\'evy process) and $X$ a L\'evy process with values in $\mathbb{R}^d$. Assume that $(X_t)_{t \geq 0}$ and $(Z_t)_{t \geq 0}$ are independent. According to \cite[Thm. 30.1]{Sato}, the process $Y$ defined by
\begin{align*}
Y_t(\omega) = X_{Z_t(\omega)}(\omega), \quad t \geq 0
\end{align*}
is a L\'evy process on $\mathbb{R}^d$. The process $(Y_t)_{t \geq 0}$ is said to be \textit{subordinate} to $(X_t)_{t \geq 0}$. Letting $\lambda = \mathcal{L}(Z_1)$ and $\mu = \mathcal{L}(X_1)$, we define the \textit{mixture} $\mu \circ \lambda := \mathcal{L}(Y_1)$. If $\mu$ is a Normal distribution, $\mu \circ \lambda$ is called a \textit{Normal variance-mean mixture} (cf. \cite{Soerensen}), and the process $Y$ is called a \textit{time-changed Brownian motion}.

The characteristic function of $\mu \circ \lambda$ is, according to \cite[Thm. 30.1]{Sato},
\begin{align}\label{cf-mixture}
\varphi_{\mu \circ \lambda} = L_{\lambda}(\log \hat{\mu}(z)), \quad z \in \mathbb{R}^d
\end{align}
where $L_{\lambda}$ denotes the Laplace transform
\begin{align*}
L_{\lambda}(w) = \int_0^{\infty} e^{wx} \lambda(dx), \quad \text{$w \in \mathbb{C}$ with $\Re w \leq 0$}
\end{align*}
and where $\log \hat{\mu}$ denotes the unique continuous logarithm of the characteristic function of $\mu$ \cite[Lemma 7.6]{Sato}.

\textit{Generalized hyperbolic distributions} $GH(\lambda,\alpha,\beta,\delta,\mu)$ with drift $\mu = 0$ are Normal variance-mean mixtures, because (see, e.g., \cite{Eberlein-Hammerstein})
\begin{align}\label{gh-mix}
GH(\lambda,\alpha,\beta,\delta,0) = N(\beta,1) \circ GIG(\lambda,\delta,\sqrt{\alpha^2 - \beta^2}),
\end{align}
where GIG denotes the \textit{generalized inverse Gaussian distribution}. For GIG-distributions it holds the convergence
\begin{align}\label{gig-conv}
GIG(\lambda,\delta,\gamma) \overset{w}{\rightarrow} {\textstyle \Gamma(\lambda,\frac{\gamma^2}{2})} \quad \text{as $\delta \downarrow 0$,} 
\end{align}
see, e.g., \cite[Sec. 5.3.5]{Schoutens}.

The characteristic function of a \textit{Variance Gamma distribution} $VG(\mu,\sigma^2,\nu)$ is (see \cite[Sec. 6.1.1]{Madan}) given by
\begin{align}\label{cf-vg}
\phi(z) = \left( 1 - iz \mu \nu + \frac{\sigma^2 \nu}{2} z^2 \right)^{-\frac{1}{\nu}}, \quad z \in \mathbb{R}.
\end{align}
Hence, we verify by using (\ref{cf-mixture}) that Variance Gamma distributions are Normal variance-mean mixtures, namely it holds
\begin{align}\label{vg-mix}
VG(\mu,\sigma^2,\nu) = N(\mu,\sigma^2) \circ {\textstyle \Gamma(\frac{1}{\nu},\frac{1}{\nu})} = {\textstyle N(\frac{\mu}{\sigma^2},1) \circ \Gamma(\frac{1}{\nu},\frac{1}{\nu \sigma^2})}.
\end{align}
It follows from \cite[Sec. 6.1.3]{Madan} that Variance Gamma distributions are special cases of bilateral Gamma distributions. In Theorem \ref{thm-variance-gamma} we characterize those bilateral Gamma distributions which are Variance Gamma. Before, we need an auxiliary result about the convergence of mixtures.

\begin{lemma}\label{lemma-mix-conv}
$\lambda_n \overset{w}{\rightarrow} \lambda$ and $\mu_n \overset{w}{\rightarrow} \mu$ implies that $\lambda_n \circ \mu_n \overset{w}{\rightarrow} \lambda \circ \mu$ as $n \rightarrow \infty$.
\end{lemma}

\begin{proof}
Fix $z \in \mathbb{R}^d$. Since $\log \hat{\mu}_n \rightarrow \log \hat{\mu}$ \cite[Lemma 7.7]{Sato}, the set
\begin{align*}
K := \{ \log \hat{\mu}_n(z) : n \in \mathbb{N} \} \cup \{ \log \hat{\mu}(z) \}
\end{align*}
is compact. It holds $L_{\lambda_n} \rightarrow L_{\lambda}$ uniformly on compact sets (the proof is analogous to that of L\'evy's Continuity Theorem). Taking into account (\ref{cf-mixture}), we thus obtain $\varphi_{\lambda_n \circ \mu_n}(z) \rightarrow \varphi_{\lambda \circ \mu}(z)$ as $n \rightarrow \infty$.
\end{proof}

Now we formulate and prove the announced theorem.

\begin{theorem}\label{thm-variance-gamma}
Let $\alpha^+,\lambda^+,\alpha^-,\lambda^- > 0$ and $\gamma = \Gamma(\alpha^+,\lambda^+;\alpha^-,\lambda^-)$. There is
equivalence between:
\begin{enumerate}
\item $\gamma$ is a
Variance Gamma distribution.

\item $\gamma$ is a
limiting case of $GH(\lambda,\alpha,\beta,\delta,0)$, where
$\delta \downarrow 0$, and $\lambda,\alpha,\beta$ are fixed.

\item $\gamma$ is a Normal variance-mean mixture.

\item $\alpha^+ = \alpha^-$.
\end{enumerate}
\end{theorem}

\begin{proof}
Assume $\gamma = VG(\mu,\sigma^2,\nu)$. We set
\begin{align*}
(\lambda,\alpha,\beta) := \left( \frac{1}{\nu}, \sqrt{\frac{2}{\nu
\sigma^2} + \left( \frac{\mu}{\sigma^2} \right)^2},
\frac{\mu}{\sigma^2} \right),
\end{align*}
and obtain by using (\ref{gh-mix}), Lemma \ref{lemma-mix-conv}, (\ref{gig-conv}) and (\ref{vg-mix})
\begin{align*}
GH(\lambda,\alpha,\beta,\delta,0) &= N(\beta,1) \circ
GIG(\lambda,\delta,\sqrt{\alpha^2 - \beta^2})
= {\textstyle N \left( \frac{\mu}{\sigma^2},1 \right) \circ
GIG \left( \frac{1}{\nu},\delta,\sqrt{\frac{2}{\nu \sigma^2}} \right)}
\\ &\overset{w}{\rightarrow} N {\textstyle \left( \frac{\mu}{\sigma^2},1 \right)
\circ \Gamma \left( \frac{1}{\nu},\frac{1}{\nu
\sigma^2} \right)} = \gamma \quad
\text{as $\delta \downarrow 0$,}
\end{align*}
showing (1) $\Rightarrow$ (2). If $GH(\lambda,\alpha,\beta,\delta,0) = N(\beta,1) \circ
GIG(\lambda,\delta,\alpha^2 - \beta^2) \overset{w}{\rightarrow} \gamma$ for $\delta \downarrow 0$, then $\gamma$ is a Normal variance-mean mixture by Lemma \ref{lemma-mix-conv}, which proves (2) $\Rightarrow$ (3). The implication (3) $\Rightarrow$ (4) is valid by \cite[Prop. 4.1]{Cont-Tankov}. If $\alpha^+ = \alpha^- =: \alpha$, using the characteristic functions (\ref{cf-bilateral}), (\ref{cf-vg}) we obtain that $\gamma = VG(\mu,\sigma^2,\nu)$ with parameters
\begin{align}\label{para-for-vg}
(\mu,\sigma^2,\nu) := \left( \frac{\alpha}{\lambda^+} -
\frac{\alpha}{\lambda^-}, \frac{2 \alpha} {\lambda^+ \lambda^-},
\frac{1}{\alpha} \right),
\end{align} 
whence (4) $\Rightarrow$ (1) follows.
\end{proof}

We emphasize that bilateral Gamma distributions which are not
Variance Gamma cannot be obtained as limiting case of generalized
hyperbolic distributions. We refer to \cite{Eberlein-Hammerstein},
where all limits of generalized
hyperbolic distributions are determined.

\section{Statistics of bilateral Gamma
distributions}\label{sec-statistics-distributions}

The results of the previous sections show that bilateral Gamma distributions have a series
of properties making them interesting for applications.

Assume we have a set of data, and suppose its law actually is a bilateral Gamma distribution.
Then we need to estimate the parameters. This section is devoted to the statistics of
bilateral Gamma distributions.

Let $X_1,\ldots,X_n$ be an i.i.d. sequence of
$\Gamma(\Theta)$-distributed random
variables, where $\Theta = (\alpha^+,\alpha^-,\lambda^+,\lambda^-)$, and let $x_1,\ldots,x_n$ be a realization. 
We would like to find an estimation $\hat{\Theta}$ of the parameters. We start with the \textit{method of moments}
and estimate the $k$-th moments $m_k = \mathbb{E} [X_1^k]$ for $k=1,\ldots,4$ as
\begin{align}\label{est-moments}
\hat{m}_k = \frac{1}{n} \sum_{i=1}^n x_i^k.
\end{align}
By \cite[p. 346]{Lexikon}, the following relations between the
moments and the cumulants $\kappa_1, \ldots, \kappa_4$ in (\ref{cumulant-kappa-gamma}) are valid:
\begin{align}\label{relation-mom-cum}
\left\{ \begin{array}{rcl} \kappa_1 & = & m_1 \medskip
\\ \kappa_2 & = & m_2 - m_1^2 \medskip
\\ \kappa_3 & = & m_3 - 3m_1m_2 + 2m_1^3 \medskip
\\ \kappa_4 & = & m_4 - 4m_3m_1 - 3m_2^2 + 12m_2m_1^2 - 6m_1^4 \end{array} \right. .
\end{align}
Inserting the cumulants (\ref{cumulant-kappa-gamma}) for $n=1,\ldots,4$ into
(\ref{relation-mom-cum}), we obtain
\begin{align}\label{system-for-method-of-m}
\left\{ \begin{array}{rcl} \alpha^+ \lambda^- - \alpha^- \lambda^+ - c_1 \lambda^+ \lambda^- & = & 0 \medskip
\\ \alpha^+ (\lambda^-)^2 + \alpha^- (\lambda^+)^2 - c_2 (\lambda^+)^2 (\lambda^-)^2 & = & 0 \medskip
\\ \alpha^+ (\lambda^-)^3 - \alpha^- (\lambda^+)^3 - c_3 (\lambda^+)^3 (\lambda^-)^3 & = & 0 \medskip
\\ \alpha^+ (\lambda^-)^4 + \alpha^- (\lambda^+)^4 - c_4 (\lambda^+)^4 (\lambda^-)^4 & = & 0
\end{array} \right. ,
\end{align}
where the constants $c_1, \ldots, c_4$ are given by
\begin{align*}
\left\{ \begin{array}{rcl} c_1 & = & m_1 \medskip
\\ c_2 & = & m_2 - m_1^2 \medskip
\\ c_3 & = & \frac{1}{2}m_3 - \frac{3}{2}m_1m_2 + m_1^3 \medskip
\\ c_4 & = & \frac{1}{6}m_4 - \frac{2}{3}m_3m_1 - \frac{1}{2}m_2^2 + 2m_2m_1^2 - m_1^4 \end{array} \right. .
\end{align*}
We can solve the system of equations (\ref{system-for-method-of-m}) explicitly. In general, if we avoid the trivial cases $(\alpha^+,\lambda^+) = (0,0)$, $(\alpha^-,\lambda^-) = (0,0)$ and $(\lambda^+,\lambda^-) = (0,0)$, it has finitely many, but more than one solution. Notice that with $(\alpha^+,\alpha^-,\lambda^+,\lambda^-)$ the vector $(\alpha^-,\alpha^+,-\lambda^-,-\lambda^+)$ is also a solution of (\ref{system-for-method-of-m}).
However, in practice, see e.g. Section \ref{sec-illustration}, the restriction $\alpha^+, \alpha^-, \lambda^+, \lambda^- > 0$ ensures uniqueness of the solution.

Let us have a closer look at the system of equations (\ref{system-for-method-of-m}) concerning solvability and uniqueness of solutions. Of course, the true values of $\alpha^+, \alpha^-, \lambda^+, \lambda^- > 0$ solve (\ref{system-for-method-of-m}) if the $c_n$ are equal to $\frac{\kappa_n}{(n-1)!}$, see (\ref{cumulant-kappa-gamma}). The left-hand side of (\ref{system-for-method-of-m}) defines a smooth function $G : C \times (0,\infty)^4 \rightarrow \mathbb{R}^4$, where $C := (\mathbb{R} \times (0,\infty))^2$. Now we consider
\begin{align}\label{eqn-as-implicit}
G(c,\vartheta) = 0, \quad \vartheta = (\alpha^+,\alpha^-,\lambda^+,\lambda^-) \in (0,\infty)^4
\end{align}
with the fixed vector $c = (c_1,\ldots,c_4)$ given by $c_n = \frac{\kappa_n}{(n-1)!}$ for $n = 1,\ldots,4$. Because of
\begin{align*}
\det \frac{\partial G}{\partial \vartheta}(c,\vartheta) &= \det
\left(
\begin{array}{cccc}
\lambda^- & -\lambda^+ & -\alpha^+ \lambda^- / \lambda^+ & \alpha^- \lambda^+ / \lambda^-
\\  (\lambda^-)^2 & (\lambda^+)^2 & -2\alpha^+ (\lambda^-)^2 / \lambda^+ & -2\alpha^- (\lambda^+)^2 / \lambda^-
\\  (\lambda^-)^3 & -(\lambda^+)^3 & -3\alpha^+ (\lambda^-)^3 / \lambda^+ & 3\alpha^- (\lambda^+)^3 / \lambda^-
\\  (\lambda^-)^4 & (\lambda^+)^4 & -4\alpha^+ (\lambda^-)^4 / \lambda^+ & -4\alpha^- (\lambda^+)^4 / \lambda^-
\end{array}
\right)
\\ &= \alpha^+ \alpha^- \lambda^+ \lambda^- \cdot \det \left(
\begin{array}{cccc}
1 & 1 & 1 & 1
\\ \lambda^- & (-\lambda^+) & 2 \lambda^- & 2(-\lambda^+)
\\ (\lambda^-)^2 & (-\lambda^+)^2 & 3 (\lambda^-)^2 & 3(-\lambda^+)^2
\\ (\lambda^-)^3 & (-\lambda^+)^3 & 4 (\lambda^-)^3 & 4 (-\lambda^+)^3
\end{array}
\right)
\\ &= \alpha^+ \alpha^- (\lambda^+)^2 (\lambda^-)^2 (\lambda^+ + \lambda^-)^4 > 0 
\end{align*}
for each $\vartheta \in (0, \infty)^4$, equation (\ref{eqn-as-implicit})
defines implicitely in a neighborhood $U$ of $c$ a uniquely defined function
$\vartheta = \vartheta (\gamma) $ with $G(\gamma, \vartheta (\gamma)) =
0$, $\gamma \in U$. Assuming the $\hat{c}_n$ calculated on the basis
of $\hat{m}_n$ are near the true $c_n$, we get a unique solution of
(\ref{system-for-method-of-m}).

This procedure yields a vector $\hat{\Theta}_0$ as first estimation for the parameters.   
Bilateral Gamma distributions are absolutely continuous with
respect to the Lebesgue measure, because they are the convolution of two Gamma distributions. In order to perform a \textit{maximum likelihood estimation}, we need adequate representations of their density functions.
Since the densities satisfy the symmetry relation
\begin{align}\label{symm-rel}
f(x;\alpha^+,\lambda^+,\alpha^-,\lambda^-) = f(-x;\alpha^-,\lambda^-,\alpha^+,\lambda^+), \quad x \in \mathbb{R} \setminus \{ 0 \}
\end{align}
it is sufficient to analyze the density functions on the positive real line.
As the convolution of two Gamma densities, they are for $x \in (0,\infty)$ given by
\begin{align}\label{density-bilateral-conv}
f(x) = \frac{(\lambda^+)^{\alpha^+}
(\lambda^-)^{\alpha^-}}{(\lambda^+ + \lambda^-)^{\alpha^-}
\Gamma(\alpha^+) \Gamma(\alpha^-)} e^{-\lambda^+ x} \int_0^\infty
v^{\alpha^- - 1} \left( x + \frac{v}{\lambda^+ + \lambda^-}
\right)^{\alpha^+ - 1} e^{-v} dv.
\end{align}
We can express the density $f$ by means of the \textit{Whittaker function} $W_{\lambda,\mu}(z)$
\cite[p. 1014]{Gradstein}, which is a well-studied mathematical function. According to \cite[p. 1015]{Gradstein}, the Whittaker function has the representation
\begin{align}\label{repr-whittaker}
W_{\lambda,\mu}(z) = \frac{z^{\lambda} e^{-\frac{z}{2}}}{\Gamma(\mu-\lambda+\frac{1}{2})}
\int_0^{\infty} t^{\mu - \lambda - \frac{1}{2}} e^{-t} 
\left( 1 + \frac{t}{z} \right)^{\mu + \lambda - \frac{1}{2}} dt
\quad \text{for $\mu - \lambda > -\frac{1}{2}$.}
\end{align}
From (\ref{density-bilateral-conv}) and (\ref{repr-whittaker}) we obtain for $x > 0$
\begin{align}\label{repr-dens-whittaker}
f(x) &= \frac{(\lambda^+)^{\alpha^+} (\lambda^-)^{\alpha^-}}{(\lambda^+ + \lambda^-)^{\frac{1}{2}(\alpha^+ + \alpha^-)}
\Gamma(\alpha^+)} x^{\frac{1}{2}(\alpha^+ + \alpha^-) - 1} 
e^{-\frac{x}{2}(\lambda^+ - \lambda^-)} 
\\ \notag & \quad \times W_{\frac{1}{2}(\alpha^+ - \alpha^-), \frac{1}{2}(\alpha^+ + \alpha^- - 1)}
(x(\lambda^+ + \lambda^-)).
\end{align}
The logarithm of the likelihood function for $\Theta = (\alpha^+,\alpha^-,\lambda^+,\lambda^-)$ is, by the symmetry relation (\ref{symm-rel}) and the representation (\ref{repr-dens-whittaker}) of the density, given by
\begin{align}\label{log-likelihood-fkt}
\ln L(\Theta) &= - n^+ \ln(\Gamma(\alpha^+)) - n^- \ln(\Gamma(\alpha^-))
\\ \notag & \quad + n \left( \alpha^+ \ln(\lambda^+) + \alpha^- \ln(\lambda^-) 
- \frac{\alpha^+ + \alpha^-}{2}  \ln(\lambda^+ + \lambda^-) \right)
\\ \notag & \quad + \left( \frac{\alpha^+ + \alpha^-}{2} - 1 \right)
\left( \sum_{i=1}^n \ln|x_i| \right) - \frac{\lambda^+ - \lambda^-}{2} \left( \sum_{i=1}^n x_i \right)
\\ \notag & \quad + \sum_{i=1}^n \ln \left( W_{\frac{1}{2}{\rm sgn}(x_i)(\alpha^+ - \alpha^-), 
\frac{1}{2}(\alpha^+ + \alpha^- - 1)}
(|x_i|(\lambda^+ + \lambda^-)) \right),
\end{align}
where $n^+$ denotes the number of positive, and $n^-$ the number of negative observations.
We take the vector $\hat{\Theta}_0$, obtained from the method of moments, as starting point for an algorithm, for example the Hooke-Jeeves algorithm \cite[Sec. 7.2.1]{Quarteroni}, which maximizes the logarithmic likelihood function (\ref{log-likelihood-fkt}) numerically.
This gives us a maximum likelihood estimation $\hat{\Theta}$ of the parameters. We shall illustrate the whole procedure in Section \ref{sec-illustration}.

\section{Bilateral Gamma processes}\label{sec-processes}

As we have shown in Section \ref{sec-distribution}, bilateral Gamma
distributions are infinitely divisible. Let us list the properties of the associated L{\'e}vy 
processes, which are denoted by $X$ in the sequel.

From the representation (\ref{levy-measure}) of the L\'evy measure $F$ we see that $F(\mathbb{R}) = \infty$ and $\int_{\mathbb{R}} |x| F(dx) < \infty$. Since the Gaussian part is zero, $X$ is of type B in the terminology of \cite[Def. 11.9]{Sato}. We obtain the following properties. Bilateral Gamma processes are \textit{finite-variation processes} \cite[Thm. 21.9]{Sato} making infinitely many jumps at each interval with positive length \cite[Thm. 21.3]{Sato}, and they are equal to the sum of their jumps \cite[Thm. 19.3]{Sato}, i.e.
\begin{align*}
X_t = \sum_{s \leq t} \Delta X_s = x * \mu^X, \quad t \geq 0
\end{align*}
where $\mu^X$ denotes the random measure of jumps of $X$. Bilateral Gamma processes are \textit{special semimartingales} with canonical decomposition \cite[Cor. II.2.38]{Jacod-Shiryaev} 
\begin{align*}
X_t = x * (\mu^X - \nu)_t + \left( \frac{\alpha^+}{\lambda^+} - \frac{\alpha^-}{\lambda^-} \right) t, \quad t \geq 0
\end{align*}
where $\nu$ is the compensator of $\mu^X$, which is given by $\nu(dt,dx) = dt F(dx)$ with $F$ denoting the L\'evy measure given by (\ref{levy-measure}).

We immediately see from the characteristic function (\ref{cf-bilateral}) that
all increments of $X$ have a bilateral Gamma distribution, more precisely
\begin{align}\label{gamma-inc}
X_t - X_s \sim \Gamma(\alpha^+(t-s),\lambda^+;\alpha^-(t-s),\lambda^-) \quad \text{for $0 \leq s < t$.}
\end{align}
There are many efficient algorithms for generating Gamma random variables, for example Johnk's generator and Best's generator of Gamma variables, chosen in \cite[Sec. 6.3]{Cont-Tankov}.
By virtue of (\ref{gamma-inc}), it is therefore easy to simulate bilateral Gamma processes.

\section{Measure transformations for bilateral Gamma processes}

Equivalent changes of measure are important in order to define
arbitrage-free financial models. In this section, we deal with
equivalent measure transformations for bilateral Gamma processes.

We assume that the probability space
$(\Omega,\mathcal{F},\mathbb{P})$ is given as follows. Let $\Omega
= \mathbb{D}$, the collection of functions $\omega(t)$ from $\mathbb{R}_+$ into $\mathbb{R}$, right-continuous with left limits. For $\omega \in \Omega$, let $X_t(\omega) = \omega(t)$ and let $\mathcal{F} = \sigma(X_t : t \in \mathbb{R}_+)$ and $(\mathcal{F}_t)_{t \geq 0}$ be the filtration
$\mathcal{F}_t = \sigma(X_s : s \in [0,t])$. We consider a probability measure $\mathbb{P}$ on $(\Omega, \mathcal{F})$ such that $X$ is a bilateral Gamma process.

\begin{proposition}\label{prop-measure-transform}
Let $X$ be a $\Gamma(\alpha_1^+, \lambda_1^+;
\alpha_1^-, \lambda_1^-)$-process under the measure $\mathbb{P}$ and let $\alpha_2^+, \lambda_2^+, \alpha_2^-, \lambda_2^- > 0$.
The following two statements are equivalent.
\begin{enumerate}
\item There is another measure $\mathbb{Q} \overset{{\rm
loc}}{\sim} \mathbb{P}$ under which $X$ is a bilateral Gamma
process with parameters $\alpha_2^+, \lambda_2^+, \alpha_2^-,
\lambda_2^-$.

\item $\alpha_1^+ = \alpha_2^+$ and $\alpha_1^- = \alpha_2^-$.
\end{enumerate}
\end{proposition}

\begin{proof}
All conditions of \cite[Thm. 33.1]{Sato} are obviously satisfied, with exception of
\begin{align}\label{int-crit-measure-trans}
\int_{\mathbb{R}} \left( 1 - \sqrt{\Phi(x)} \right)^2 F_1(dx) <
\infty,
\end{align}
where $\Phi = \frac{d F_2}{d F_1}$ denotes the Radon-Nikodym derivative of the respective L\'evy measures, which is by (\ref{levy-measure}) given by
\begin{align}\label{def-phi-trans}
\Phi(x) = \frac{\alpha_2^+}{\alpha_1^+} e^{-(\lambda_2^+ - \lambda_1^+)
x} \mathbf{1}_{(0,\infty)}(x) + \frac{\alpha_2^-}{\alpha_1^-} e^{-(\lambda_2^- - \lambda_1^-)
\abs{x}} \mathbf{1}_{(-\infty,0)}(x), \quad x \in \mathbb{R}.
\end{align}
The integral in (\ref{int-crit-measure-trans}) is equal to
\begin{align*}
\int_{\mathbb{R}} \left( 1 - \sqrt{\Phi(x)} \right)^2 F_1(dx) &=
\int_0^{\infty} \frac{1}{x} \left( \sqrt{\alpha_2^+}
e^{-(\lambda_2^+/2) x} - \sqrt{\alpha_1^+} e^{-(\lambda_1^+/2) x}
\right)^2 dx
\\ & \quad + \int_0^{\infty} \frac{1}{x} \left(
\sqrt{\alpha_2^-} e^{-(\lambda_2^-/2) x} - \sqrt{\alpha_1^-}
e^{-(\lambda_1^-/2) x} \right)^2 dx.
\end{align*}
Hence, condition (\ref{int-crit-measure-trans}) is satisfied if and only if $\alpha_1^+ = \alpha_2^+$ and $\alpha_1^- = \alpha_2^-$. Applying \cite[Thm. 33.1]{Sato} completes the proof.
\end{proof}

Proposition \ref{prop-measure-transform} implies that we can transform any Variance Gamma process, which is according to Theorem \ref{thm-variance-gamma} a bilateral Gamma process $\Gamma(\alpha,\lambda^+;\alpha,\lambda^-)$, into
a symmetric bilateral Gamma process $\Gamma(\alpha,\lambda;\alpha,\lambda)$ with arbitrary
parameter $\lambda > 0$. %\bigskip

Now assume the process $X$ is
$\Gamma(\alpha^+,\lambda_1^+;\alpha^-,\lambda_1^-)$ under
$\mathbb{P}$ and
$\Gamma(\alpha^+,\lambda_2^+;\alpha^-,\lambda_2^-)$ under
the measure $\mathbb{Q} \overset{{\rm loc}}{\sim} \mathbb{P}$. According to Proposition \ref{prop-measure-transform}, such a change of measure exists.
For the computation of the \textit{likelihood process}
\begin{align*} 
\Lambda_t(\mathbb{Q},\mathbb{P}) =
\frac{d \mathbb{Q} |_{\mathcal{F}_t}}{d \mathbb{P}
|_{\mathcal{F}_t}}, \quad t \geq 0
\end{align*}
we will need the following auxiliary result Lemma \ref{lemma-integral-help}. For its proof, we require the following properties of the \textit{Exponential Integral} \cite[Chap. 5]{Abramowitz}
\begin{align*}
E_1(x) := \int_1^{\infty} \frac{e^{-xt}}{t} dt, \quad x > 0.
\end{align*}
The Exponential Integral has the series expansion
\begin{align}\label{series-ei}
E_1(x) = - \gamma - \ln x - \sum_{n=1}^{\infty} \frac{(-1)^n}{n \cdot n!} x^n,
\end{align}
where $\gamma$ denotes Euler's constant
\begin{align*}
\gamma = \lim_{n \rightarrow \infty} \left[ \sum_{k=1}^n
\frac{1}{k} - \ln(n) \right].
\end{align*}
The derivative of the Exponential Integral is given by
\begin{align}\label{der-exp-int}
\frac{\partial}{\partial x}E_1(x) = - \frac{e^{-x}}{x}.
\end{align}

\begin{lemma}\label{lemma-integral-help}
For all $\lambda_1, \lambda_2 > 0$ it holds
\begin{align*}
\int_0^{\infty} \frac{e^{-\lambda_2 x} - e^{-\lambda_1
x}}{x} dx = \ln \left( \frac{\lambda_1}{\lambda_2} \right).
\end{align*}
\end{lemma}
 
\begin{proof}
Due to relation (\ref{der-exp-int}) and the series expansion (\ref{series-ei}) of the
Exponential Integral $E_1$ we obtain
\begin{align*}
&\int_0^{\infty} \frac{e^{-\lambda_2 x} - e^{-\lambda_1
x}}{x} dx = \lim_{b \rightarrow \infty} \left[ E_1(\lambda_1 b) - E_1(\lambda_2 b) \right] - \lim_{a \rightarrow 0} \left[ E_1(\lambda_1 a) - E_1(\lambda_2 a) \right]
\\ &= \lim_{b \rightarrow \infty}
E_1(\lambda_1 b) - \lim_{b \rightarrow \infty} E_1(\lambda_2 b)
\\ &\quad + \ln \left( \frac{\lambda_1}{\lambda_2} \right) + \lim_{a \rightarrow
0} \sum_{n=1}^{\infty} \frac{1}{n \cdot n!} (\lambda_1
a)^n - \lim_{a \rightarrow 0} \sum_{n=1}^{\infty} \frac{1}{n
\cdot n!} (\lambda_2 a)^n.
\end{align*}
Each of the four limits is zero, so the claimed identity follows.
\end{proof}

For our applications to finance, the \textit{relative entropy}
$\mathcal{E}_t(\mathbb{Q},\mathbb{P}) = \mathbb{E}_{\mathbb{Q}}
[\ln \Lambda_t(\mathbb{Q},\mathbb{P})]$, also known as \textit{Kullback-Leibler distance}, which is often used as
measure of proximity of two equivalent probability measures, will be of importance. The upcoming result provides the likelihood process and the relative entropy.
In the degenerated cases $\lambda_1^+ = \lambda_2^+$ or
$\lambda_1^- = \lambda_2^-$, the associated Gamma distributions in (\ref{deg-gamma-delta}) are
understood to be the Dirac measure $\delta(0)$.

\begin{proposition}\label{density-repr}
It holds $\Lambda_t(\mathbb{Q},\mathbb{P}) = e^{U_t}$, where $U$
is under $\mathbb{P}$ the L$\acute{e}$vy process with generating distribution
\begin{align}\label{deg-gamma-delta}
U_1 \sim \Gamma \left( \alpha^+, \frac{\lambda_1^+}{\lambda_1^+ -
\lambda_2^+}\right) * \Gamma \left( \alpha^-,
\frac{\lambda_1^-}{\lambda_1^- - \lambda_2^-}\right) * \delta
\left( \alpha^+ \ln \left( \frac{\lambda_2^+}{\lambda_1^+}
\right) + \alpha^- \ln \left( \frac{\lambda_2^-}{\lambda_1^-}
\right) \right).
\end{align}
Setting $f(x) = x - 1 - \ln x$, it holds for the relative entropy
\begin{align}\label{rel-entropy}
\mathcal{E}_t(\mathbb{Q},\mathbb{P}) = t \left[ \alpha^+
f \left( \frac{\lambda_1^+}{\lambda_2^+} \right) + \alpha^-
f \left( \frac{\lambda_1^-}{\lambda_2^-} \right) \right].
\end{align}
\end{proposition}

\begin{proof}
According to \cite[Thm. 33.2]{Sato}, the likelihood process is of the form  $\Lambda_t(\mathbb{Q},\mathbb{P}) = e^{U_t}$, where $U$
is, under the measure $\mathbb{P}$, the L\'evy process
\begin{align} \label{proc-Ut}
U_t = \sum_{s \leq t} \ln (\Phi(\Delta X_s)) - t \int_{\mathbb{R}}
(\Phi(x) - 1) F_1(dx),
\end{align}
and where $\Phi$ is the Radon-Nikodym derivative given by (\ref{def-phi-trans}) with $\alpha_1^+ = \alpha_2^+ =: \alpha^+$ and $\alpha_1^- = \alpha_2^- =: \alpha^-$. For every $t > 0$ denote by $X_t^+$ the sum $\sum_{s \leq t}(\Delta X_s)^+$ and by $X_t^-$ the sum $\sum_{s \leq t}(\Delta X_s)^-$. Then $X = X^+ - X^-$. By construction and the definition of $\mathbb{Q}$, the processes $X^+$ and $X^-$ are independent $\Gamma(\alpha^+,\lambda_1^+)$- and $\Gamma(\alpha^-, \lambda_1^-)$-processes under $\mathbb{P}$ and independent $\Gamma(\alpha^+,\lambda_2^+)$- and $\Gamma(\alpha^-, \lambda_2^-)$-processes under $\mathbb{Q}$, respectively. From (\ref{def-phi-trans}) it follows
\begin{align*}
\sum_{s \leq t} \ln (\Phi(\Delta X_s)) = (\lambda_1^+ -
\lambda_2^+) X_t^+ + (\lambda_1^- - \lambda_2^-) X_t^-.
\end{align*}
The integral in (\ref{proc-Ut}) is, by using Lemma \ref{lemma-integral-help}, equal to
\begin{align*}
\int_{\mathbb{R}} (\Phi(x) - 1) F_1(dx) &= \alpha^+ \int_0^{\infty} \frac{e^{-\lambda_2^+ x} - e^{-\lambda_1^+
x}}{x} dx + \alpha^- \int_0^{\infty} \frac{e^{-\lambda_2^- x} - e^{-\lambda_1^- x}}{x} dx
\\ &= \alpha^+ \ln \left( \frac{\lambda_1^+}{\lambda_2^+} \right)
+ \alpha^- \ln \left( \frac{\lambda_1^-}{\lambda_2^-} \right).
\end{align*}
Hence, we obtain
\begin{align}\label{ln-likelihood-proc}
U_t = (\lambda_1^+ - \lambda_2^+)
X_t^+ + (\lambda_1^- - \lambda_2^-) X_t^- + \left[ \alpha^+ \ln
\left( \frac{\lambda_2^+}{\lambda_1^+} \right) + \alpha^- \ln
\left( \frac{\lambda_2^-}{\lambda_1^-} \right) \right] t.
\end{align}
Equation (\ref{ln-likelihood-proc}) yields (\ref{rel-entropy}) and, together with Lemma \ref{sums-of-bilateral}, the relation (\ref{deg-gamma-delta}).
\end{proof}

Since the likelihood process is of the form $\Lambda_t(\mathbb{Q},\mathbb{P}) = e^{U_t}$, where the L\'evy process $U$ is given by (\ref{ln-likelihood-proc}), one verifies that
\begin{align}\label{exp-family}
\Lambda_t(\mathbb{Q},\mathbb{P}) &= \exp \Big( (\lambda_1^+ - \lambda_2^+)X_t^+ - t \Psi^+(\lambda_1^+ - \lambda_2^+) \Big)
\\ \notag & \quad \times \exp \Big( (\lambda_1^- - \lambda_2^-)X_t^- - t \Psi^-(\lambda_1^- - \lambda_2^-) \Big),
\end{align}
where $\Psi^+, \Psi^-$ denote the respective cumulant generating functions of the Gamma processes $X^+, X^-$ under the measure $\mathbb{P}$.

Keeping $\alpha^+, \alpha^-, \lambda_1^+, \lambda_1^-$ all positive and fixed, then by putting $\vartheta^+ = \lambda_1^+ - \lambda_2^+$, $\vartheta^- = \lambda_1^- - \lambda_2^-$, $\vartheta = (\vartheta^+, \vartheta^-)^{\top} \in (-\infty, \lambda_1^+) \times (-\infty, \lambda_1^-) =: \Theta$, $\mathbb{Q} = \mathbb{Q}_{\vartheta}$ we obtain a two-parameter \textit{exponential family} $(\mathbb{Q}_{\vartheta}, \vartheta \in \Theta)$ of L\'evy processes in the sense of \cite[Chap. 3]{Kuechler-Soerensen}, with the canonical process $B_t = (X_t^+, X_t^-)$.

In particular, it follows that for every $t > 0$ the vector $B_t$ is a sufficient statistics for $\vartheta = (\vartheta^+, \vartheta^-)^{\top}$ based on the observation of $(X_s, s \leq t)$. Considering the subfamily obtained by $\vartheta^+ = \vartheta^-$, we obtain a one-parametric exponential family of L\'evy processes with $X_t^+ + X_t^- = \sum_{s \leq t} |\Delta X_s|$ as sufficient statistics and canonical process.

\section{Inspecting a typical path}

Proposition \ref{prop-measure-transform} of the previous section
suggests that the parameters $\alpha^+, \alpha^-$ should be
determinable by inspecting a typical path of a bilateral Gamma process. This is indeed the case. We start with
Gamma processes. Let $X$ be a $\Gamma(\alpha,\lambda)$-process.
Choose a finite time horizon $T > 0$ and set 
\begin{align*}
S_n := \frac{1}{nT} \# \left\{ t \leq T : \Delta X_t \geq e^{-n} \right\}, \quad n \in
\mathbb{N}.
\end{align*}

\begin{theorem}\label{thm-det-alpha}
It holds $\mathbb{P} \left( \lim_{n \rightarrow \infty} S_n = \alpha \right) = 1$.
\end{theorem}

\begin{proof}
Due to \cite[Thm. 19.2]{Sato}, the random measure $\mu^X$ of the
jumps of $X$ is a Poisson random measure with intensity measure
\begin{align*}
\nu(dt,dx) = dt \frac{\alpha e^{-\lambda x}}{x}
\mathbf{1}_{(0,\infty)} dx.
\end{align*} 
Thus, the sequence
\begin{align*}
Y_n := \frac{1}{T} \mu^X
\left( [0,T] \times [e^{-n},e^{1-n}) \right), \quad n \in \mathbb{N}
\end{align*}
defines a sequence of independent random variables with
\begin{align*}
\mathbb{E}[Y_n] &= \alpha \int_{e^{-n}}^{e^{1-n}}
\frac{e^{-\lambda x}}{x} dx = \alpha \int_{n-1}^n \exp \left(
-\lambda e^{-v} \right) dv \uparrow \alpha \quad \text{as $n
\rightarrow \infty$,}
\\ {\rm Var}[Y_n] &= \frac{\alpha}{T} \int_{e^{-n}}^{e^{1-n}}
\frac{e^{-\lambda x}}{x} dx = \frac{\alpha}{T} \int_{n-1}^n \exp
\left( -\lambda e^{-v} \right) dv \uparrow \frac{\alpha}{T} \quad
\text{as $n \rightarrow \infty$,}
\end{align*}
because $\exp \left( -\lambda e^{-v} \right) \uparrow 1$ for $v
\rightarrow \infty$. Hence, we have
\begin{align*} 
 \sum_{n=1}^{\infty} \frac{{\rm
Var}[Y_n]}{n^2} < \infty.
\end{align*} 
We may therefore apply Kolmogorov's
strong law of large numbers \cite[Thm. IV.3.2]{Shiryaev}, and
deduce that
\begin{align*}
\lim_{n \rightarrow \infty} S_n = \lim_{n
\rightarrow \infty} \frac{1}{n} \sum_{k=1}^n Y_k +
\lim_{n \rightarrow \infty} \frac{1}{nT}
\mu^X([0,T] \times [1,\infty)) = \alpha,
\end{align*}
finishing the proof.
\end{proof}

Now, let $X$ be a bilateral Gamma process, say $X_1 \sim \Gamma(\alpha^+,\lambda^+;\alpha^-,\lambda^-)$. 
We set 
\begin{align*}
S_n^+ &:= \frac{1}{nT} \# \left\{ t \leq T : \Delta X_t \geq e^{-n} \right\}, \quad n \in \mathbb{N},
\\ S_n^- &:= \frac{1}{nT} \# \left\{ t \leq T : \Delta X_t \leq -e^{-n} \right\}, \quad n \in \mathbb{N}.
\end{align*}

\begin{corollary}
It holds $\mathbb{P}(\lim_{n \rightarrow \infty} S_n^+ = \alpha^+ \text{ and } \lim_{n \rightarrow \infty} S_n^- = \alpha^-) = 1$.
\end{corollary}

\begin{proof}
We define the processes $X^+$ and $X^-$ as $X_t^+ = \sum_{s \leq t}(\Delta X_s)^+$ and $X_t^- = \sum_{s \leq t}(\Delta X_s)^-$. By construction we have $X = X^+ - X^-$ and the processes $X^+$ and $X^-$ are independent $\Gamma(\alpha^+,\lambda^+)$- and $\Gamma(\alpha^-, \lambda^-)$-processes. Applying Theorem \ref{thm-det-alpha} yields the desired result.
\end{proof}

\section{Stock models}\label{sec-stock}

We move on to present some applications to finance of the theory developed above.
Assume that the evolution of an asset price is described by an exponential L\'evy model $S_t = S_0 e^{rt + X_t}$, where $S_0 > 0$ is the (deterministic) initial value of the stock, $r$ the interest rate and where the L\'evy process $X$ is a bilateral Gamma process $\Gamma(\alpha^+,\lambda^+;\alpha^-,\lambda^-)$ under the measure $\mathbb{P}$, which plays the role of the real-world measure.

In order to avoid arbitrage, it arises the question
whether there exists an \textit{equivalent martingale measure}, i.e. a measure $\mathbb{Q} \overset{{\rm
loc}}{\sim} \mathbb{P}$ such that $Y_t := e^{-rt}S_t$ is a local martingale.

\begin{lemma}\label{measure-trans-gamma}
Assume $\lambda^+ > 1$. Then $Y$ is a local $\mathbb{P}$-martingale if and only if
\begin{align}\label{mart-measure-cond}
\left( \frac{\lambda^+}{\lambda^+ - 1} \right)^{\alpha^+} = \left( \frac{\lambda^- + 1}{\lambda^-} \right)^{\alpha^-}.
\end{align}
\end{lemma}

\begin{proof}
Since the Gaussian part of the bilateral Gamma process $X$ is zero,
It\^o's formula \cite[Thm. I.4.57]{Jacod-Shiryaev}, applied on $Y_t = S_0 e^{X_t}$, yields for the discounted stock prices
\begin{align*}
Y_t = Y_0 + \int_0^t Y_{s-} dX_s + S_0 \sum_{0 < s \leq t} \Big(
e^{X_s} - e^{X_{s-}} - e^{X_{s-}} \Delta X_s \Big).
\end{align*}
Recall from Section \ref{sec-processes} that $X = x * \mu^X$ and that the compensator $\nu$ of $\mu^X$ is given by $\nu(dt,dx) = dt F(dx)$, where $F$ denotes the L\'evy measure from (\ref{levy-measure}). So we obtain
\begin{align}\label{repr-disc-stock}
Y_t &= Y_0 + \int_0^t \int_{\mathbb{R}} x Y_{s-} \mu^X(ds,dx) +
\int_0^t \int_{\mathbb{R}} Y_{s-} ( e^x - 1 - x ) \mu^X(ds,dx)
\\ \notag &= Y_0 + \int_0^t \int_{\mathbb{R}} Y_{s-} ( e^x - 1 ) (\mu^X - \nu)(ds,dx)
+ \int_0^t Y_{s-} \int_{\mathbb{R}} ( e^x - 1 ) F(dx)ds.
\end{align}
Applying Lemma \ref{lemma-integral-help}, the integral in the drift term is equal to
\begin{align}\label{drift-term-emm}
&\int_{\mathbb{R}} ( e^x - 1 ) F(dx)
\\ \notag &= \alpha^+ \int_0^{\infty} \frac{e^{-(\lambda^+ - 1) x} - e^{-\lambda^+
x}}{x} dx - \alpha^- \int_0^{\infty} \frac{e^{-\lambda^- x} - e^{-(\lambda^- + 1) x}}{x} dx
\\ \notag &= \alpha^+ \ln \left( \frac{\lambda^+}{\lambda^+ - 1} \right) -
\alpha^- \ln \left( \frac{\lambda^- + 1}{\lambda^-} \right).
\end{align}
The discounted price process $Y$ is a local martingale if and only if the drift in (\ref{repr-disc-stock}) vanishes, and by (\ref{drift-term-emm}) this is the case if and only if (\ref{mart-measure-cond}) is satisfied.
\end{proof}

As usual in financial modelling with jump processes, the market is free of arbitrage, but not complete, that is
there exist several martingale measures. The next result shows that we can find a continuum of martingale
measures by staying within the class of bilateral Gamma processes. We define $\phi : (1,\infty) \rightarrow \mathbb{R}$ as
\begin{align*}
\phi(\lambda) := \phi(\lambda;\alpha^+,\alpha^-) := \left( \left( \frac{\lambda}{\lambda - 1}
\right)^{\alpha^+ / \alpha^-} - 1 \right)^{-1}, \quad \lambda \in (1,\infty).
\end{align*}

\begin{proposition}\label{prop-emm}
For each $\lambda \in (1,\infty)$ there exists a martingale measure $\mathbb{Q}_{\lambda} \overset{{\rm loc}}{\sim} \mathbb{P}$ such that under $\mathbb{Q}_{\lambda}$ we have
\begin{align}\label{law-of-x-under-q}
X_1 \sim \Gamma(\alpha^+,\lambda;\alpha^-,\phi(\lambda)).
\end{align}
\end{proposition}

\begin{proof}
Recall that $X$ is $\Gamma(\alpha^+,\lambda^+;\alpha^-,\lambda^-)$ under $\mathbb{P}$. According to Proposition \ref{prop-measure-transform}, for each $\lambda \in (1,\infty)$ there exists a probability measure $\mathbb{Q}_{\lambda} \overset{{\rm loc}}{\sim} \mathbb{P}$ such that under $\mathbb{Q}_{\lambda}$ relation (\ref{law-of-x-under-q}) is fulfilled, and moreover this measure $\mathbb{Q}_{\lambda}$ is a martingale measure, because equation (\ref{mart-measure-cond}) from Lemma \ref{measure-trans-gamma} is satisfied.
\end{proof}

So, we have a continuum of martingale measures, and the question is, which one we should choose. There
are several suggestions in the literature, see, e.g., \cite{Cont-Tankov}. %\bigskip

One approach is to minimize the relative entropy, which amounts
to finding $\lambda \in (1,\infty)$ which minimizes
$\mathcal{E}(\mathbb{Q}_{\lambda},\mathbb{P})$, and then taking $\mathbb{Q}_{\lambda}$.
The relative entropy is determined in equation (\ref{rel-entropy}) of Proposition \ref{density-repr}. Taking the first derivative with respect to $\lambda$, and setting it equal
to zero, we have to find the $\lambda \in (1,\infty)$ such that
\begin{align}\label{min-entropy}
\alpha^- \alpha \lambda^{\alpha - 1}
\left( \frac{1}{\lambda^{\alpha}(\lambda - 1) - (\lambda - 1)^{\alpha + 1}} 
- \frac{\lambda^-}{(\lambda - 1)^{\alpha + 1}} \right) 
+ \frac{\alpha^+}{\lambda} \left( 1 - \frac{\lambda^+}{\lambda} \right) = 0,
\end{align}
where $\alpha := \alpha^+ / \alpha^-$. This can be done numerically. %\bigskip

Another point of view is that the martingale measure is given by the market. We would like to \textit{calibrate} the L\'evy process $X$ from the family of bilateral Gamma processes to option prices. According to Proposition \ref{prop-emm} we can, by adjusting $\lambda \in (1,\infty)$, preserve the martingale property, which leaves us one parameter to calibrate.

For simplicity, we set $r=0$. For each $\lambda \in (1,\infty)$, an arbitrage free pricing rule for a \textit{European Call Option}
at time $t \geq 0$ is, provided that $S_t = s$, given by
\begin{align}\label{call-option-price}
C_{\lambda}(s,K;t,T) = \mathbb{E}_{\mathbb{Q}_{\lambda}}[(S_T - K)^+ | S_t=s], 
\end{align}
where $K$ denotes the strike price and $T > t$ the time of maturity. We can express the expectation in (\ref{call-option-price}) as
\begin{align}
\mathbb{E}_{\mathbb{Q}_{\lambda}}[(S_T - K)^+ | S_t=s] = \Pi(s,K,\alpha^+(T-t),\alpha^-(T-t),\lambda,\phi(\lambda)),
\end{align}
where $\Pi$ is defined as
\begin{align}\label{preis-integral}
\Pi(s,K,\alpha^+,\alpha^-,\lambda^+,\lambda^-) := \int_{\ln \left( \frac{K}{s} \right)}^{\infty} (se^x - K) f(x;\alpha^+,\alpha^-,\lambda^+,\lambda^-)dx,
\end{align}
with $x \mapsto f(x;\alpha^+,\alpha^-,\lambda^+,\lambda^-)$ denoting the density of a bilateral Gamma distribution having these parameters.
In order to compute the option prices, we have to evaluate the integral in (\ref{preis-integral}).
In the sequel, $F(\alpha,\beta;\gamma;z)$ denotes the \textit{hypergeometric series} \cite[p. 995]{Gradstein}
\begin{align*}
F(\alpha,\beta;\gamma;z) &= 1 + \frac{\alpha \cdot \beta}{\gamma \cdot 1} z
+ \frac{\alpha(\alpha+1)\beta(\beta+1)}{\gamma(\gamma+1) \cdot 1 \cdot 2} z^2
\\ & \quad + \frac{\alpha(\alpha+1)(\alpha+2)\beta(\beta+1)(\beta+2)}{\gamma(\gamma+1)(\gamma+2) \cdot 1 \cdot 2 \cdot 3} z^3 + \ldots
\end{align*}

\begin{proposition}\label{prop-preis-integral}
Assume $\lambda^+ > 1$. For the integral in (\ref{preis-integral}) the following identity is valid:
\begin{align}\label{pricing-formula-half-num}
&\Pi(s,K,\alpha^+,\alpha^-,\lambda^+,\lambda^-) = \int_{\ln \left( \frac{K}{s} \right)}^0 (se^x - K) f(x;\alpha^+,\alpha^-,\lambda^+,\lambda^-)dx
\\ \notag &\quad + \frac{(\lambda^+)^{\alpha^+}(\lambda^-)^{\alpha^-}\Gamma(\alpha^+ + \alpha^-)}{\Gamma(\alpha^+) \Gamma(\alpha^- + 1)}
\\ \notag &\quad \times \left( \frac{s F(\alpha^+ + \alpha^-, \alpha^-; \alpha^- + 1; - \frac{\lambda^- + 1}{\lambda^+ - 1})}{(\lambda^+ - 1)^{\alpha^+ + \alpha^-}} - \frac{K F(\alpha^+ + \alpha^-, \alpha^-; \alpha^- + 1; - \frac{\lambda^-}{\lambda^+})}{(\lambda^+)^{\alpha^+ + \alpha^-}} \right).
\end{align}
\end{proposition}

\begin{proof}
Note that the density of a bilateral Gamma distribution is given by (\ref{repr-dens-whittaker}).
The assertion follows by applying identity 3 from \cite[p. 816]{Gradstein}.
\end{proof}

Proposition \ref{prop-preis-integral} provides a closed pricing formula for exp-L\'evy models with underlying bilateral Gamma process, as the \textit{Black-Scholes formula} for Black-Scholes models. In formula (\ref{pricing-formula-half-num}), it remains to evaluate the integral over the compact interval $[\ln (\frac{K}{s}),0]$. This can be done
numerically. In the special case $K = s$ we get an exact pricing formula.

\begin{corollary}\label{cor-preis-integral}
Assume $\lambda^+ > 1$. In the case $K = s$ it holds for (\ref{preis-integral}):
\begin{align}\label{pricing-formula}
&\Pi(s,K,\alpha^+,\alpha^-,\lambda^+,\lambda^-)
= \frac{K(\lambda^+)^{\alpha^+}(\lambda^-)^{\alpha^-}\Gamma(\alpha^+ + \alpha^-)}{\Gamma(\alpha^+) \Gamma(\alpha^- + 1)} 
\\ \notag &\quad \times \left( \frac{F(\alpha^+ + \alpha^-, \alpha^-; \alpha^- + 1; - \frac{\lambda^- + 1}{\lambda^+ - 1})}{(\lambda^+ - 1)^{\alpha^+ + \alpha^-}} - \frac{F(\alpha^+ + \alpha^-, \alpha^-; \alpha^- + 1; - \frac{\lambda^-}{\lambda^+})}{(\lambda^+)^{\alpha^+ + \alpha^-}} \right).
\end{align}
\end{corollary}

\begin{proof}
This is an immediate consequence of Proposition \ref{prop-preis-integral}.
\end{proof}

We will use this result in the upcoming section in
order to calibrate our model to an option price observed at the market.

\section{An illustration: DAX 1996-1998}\label{sec-illustration}

\begin{figure}[t]
   \centering
   \includegraphics[height=30ex,width=50ex]{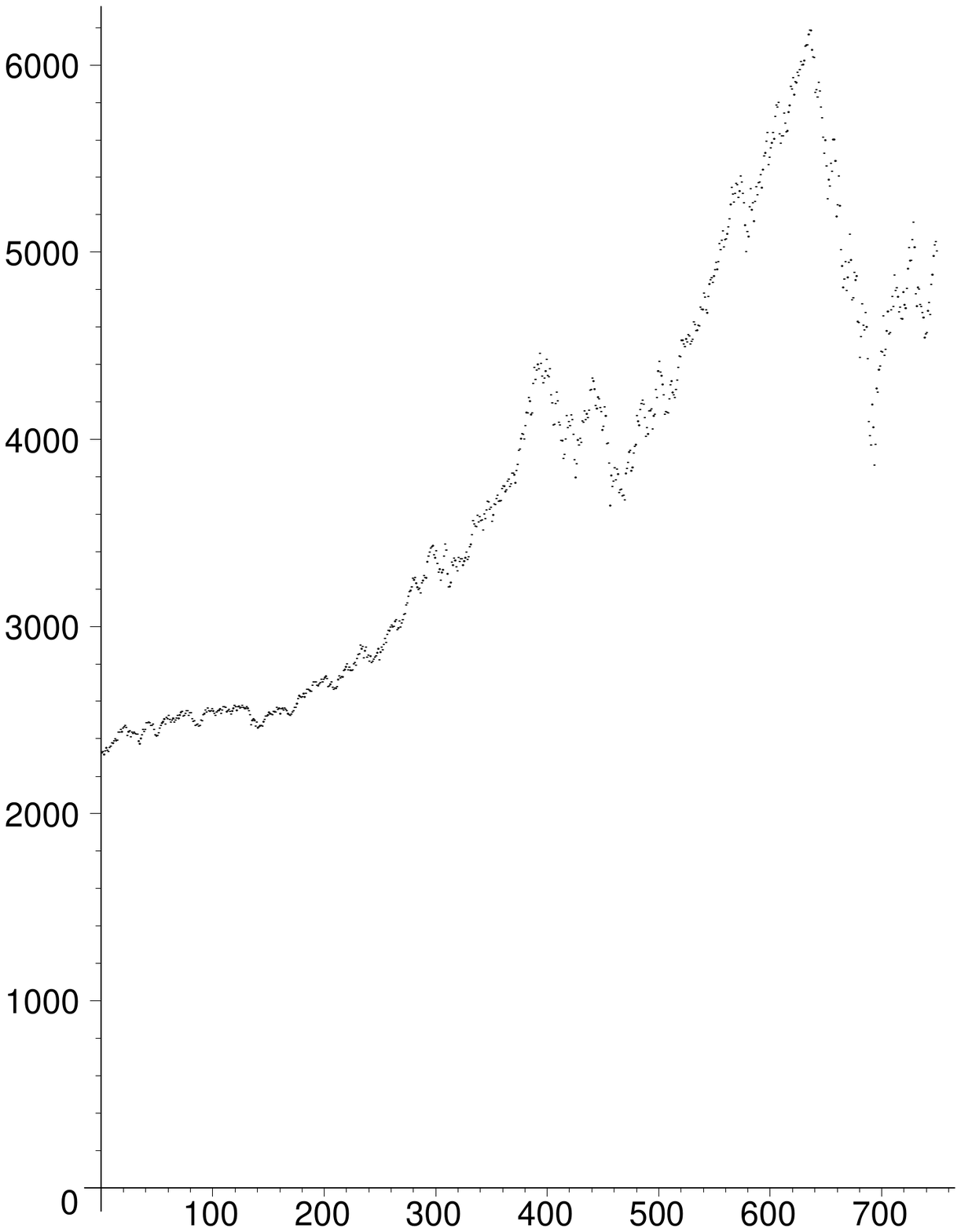}
   \caption{DAX, 1996-1998.}
   \label{fig-dax-real}
\end{figure}

We turn to an illustration of the preceding theory. Figure \ref{fig-dax-real} shows 751
observations $S_0, S_1, \ldots, S_{750}$ of the German stock index DAX, over the
period of three years. We assume that this price evolution
actually is the trajectory of an exponential bilateral Gamma
model, i.e. $S_t = S_0 e^{X_t}$ with $S_0 = 2307.7$ and $X$ being
a $\Gamma(\Theta)$-process, where $\Theta = (\alpha^+,\alpha^-,\lambda^+,\lambda^-)$. 
For simplicity we assume that the interest rate $r$ is zero.
Then the increments $\Delta X_i = X_i - X_{i-1}$ for $i=1,\ldots,750$
are a realization of an i.i.d. sequence of
$\Gamma(\Theta)$-distributed random variables. %\bigskip

In order to estimate $\Theta$, we carry out the statistical program described in Section
\ref{sec-statistics-distributions}. For the given observations $\Delta X_1, \ldots, \Delta X_{750}$, the \textit{method of moments} (\ref{est-moments}) yields the estimation
\begin{align*}
\hat{m}_1 &= 0.001032666257,
\\ \hat{m}_2 &= 0.0002100280033,
\\ \hat{m}_3 &= -0.0000008191504362,
\\ \hat{m}_4 &= 0.0000002735163873.
\end{align*}
We can solve the system of equations (\ref{system-for-method-of-m}) explicitly and obtain, apart from the trivial cases $(\alpha^+,\lambda^+) = (0,0)$, $(\alpha^-,\lambda^-) = (0,0)$ and $(\lambda^+,\lambda^-) = (0,0)$, the two solutions $(1.28, 0.78, 119.75, 80.82)$ and $(0.78 ,1.28, -80.82, -119.75)$. Taking into account the parameter condition $\alpha^+,\alpha^-,\lambda^+,\lambda^- > 0$, the system (\ref{system-for-method-of-m}) has the unique solution
\begin{align*}
\hat{\Theta}_0 = (1.28, 0.78, 119.75, 80.82).
\end{align*}
Proceeding with the Hooke-Jeeves algorithm \cite[Sec. 7.2.1]{Quarteroni}, which maximizes the logarithmic likelihood function (\ref{log-likelihood-fkt}) numerically, with $\hat{\Theta}_0$ as starting point, we obtain
the \textit{maximum likelihood estimation}
\begin{align*}
\hat{\Theta} = (1.55, 0.94, 133.96, 88.92).
\end{align*}
We have estimated the parameters of the bilateral Gamma process $X$ under the measure $\mathbb{P}$, which plays the role of the real-world measure. The next task is to find an appropriate martingale measure $\mathbb{Q}_{\lambda} \overset{{\rm loc}}{\sim} \mathbb{P}$. %\bigskip
 
Assume that at some point of time $t \geq 0$ the stock has value $S_t = 5000$ EUR, and that there is a European Call Option at the market with the same strike price $K = 5000$ EUR
and exercise time in $100$ days, i.e. $T = t + 100$. Our goal is to \textit{calibrate} our model to the price of this option. Since the stock value and the strike price coincide, we can use the exact pricing formula (\ref{pricing-formula}) from Corollary \ref{cor-preis-integral}. The resulting Figure \ref{fig-lambda-prices}
shows the Call Option prices $C_{\lambda}(5000,5000;t,t+100)$ for $\lambda \in (1,\infty)$.
Observe that we get the whole interval $(0,5000)$ of reasonable Call Option prices. This is a typical feature of exp-L\'evy models, cf. \cite{Eberlein-Jacod}.
 
Consequently, we can calibrate our model to any observed price $C \in (0,5000)$ of the Call Option by choosing the $\lambda \in (1,\infty)$ such that $C = C_{\lambda}(5000,5000;t,t+100)$. %\bigskip

\begin{figure}[t]
   \centering
   \includegraphics[height=30ex,width=50ex]{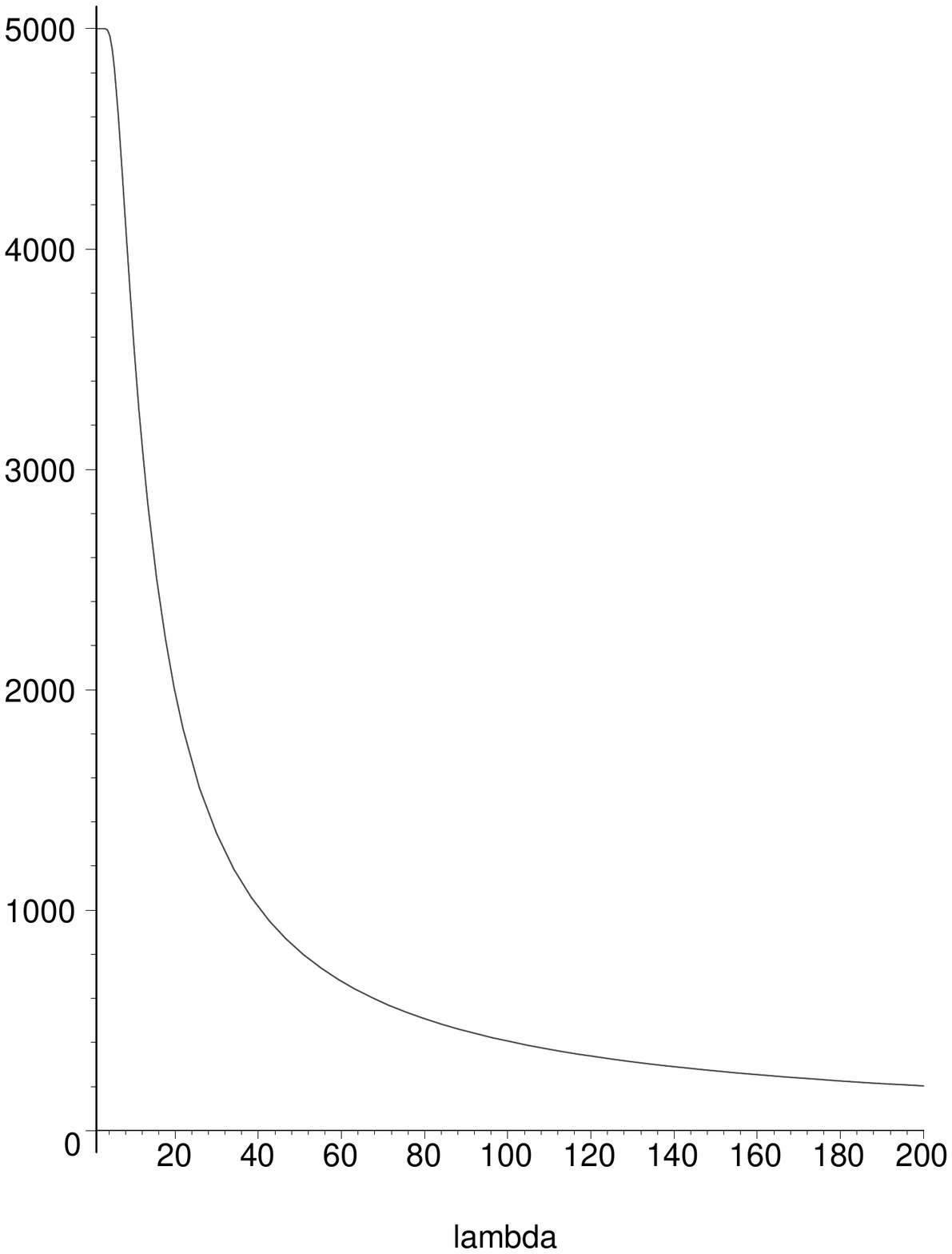}
   \caption{Call Option prices $C_{\lambda}(5000,5000;t,t+100)$ for $\lambda \in (1,\infty)$.}
   \label{fig-lambda-prices}
\end{figure}

As described in Section \ref{sec-stock}, another way to find a martingale measure is to minimize the relative entropy, i.e. finding $\lambda \in (1,\infty)$ which minimizes $\mathcal{E}(\mathbb{Q}_{\lambda},\mathbb{P})$. For this purpose, we have to find $\lambda \in (1,\infty)$ such that (\ref{min-entropy}) is satisfied. We solve this equation numerically and find the unique solution given by $\lambda = 139.47$. Using the corresponding martingale measure $\mathbb{Q}_{\lambda} \overset{{\rm loc}}{\sim} \mathbb{P}$, we obtain the Call Option price $C_{\lambda}(5000,5000;t,t+100) = 290.75$, cf. Figure \ref{fig-lambda-prices}. Under $\mathbb{Q}_{\lambda}$, the process $X$ is, according to Proposition \ref{prop-emm}, a bilateral Gamma process $\Gamma(1.55,139.47;0.94,83.51)$. %\bigskip

It remains to analyze the goodness of fit of the bilateral Gamma distribution, and to compare it to other
families of distributions. Figure \ref{fig-emp-bil-dens} shows the empirical and the fitted bilateral Gamma density.

We have provided maximum likelihood estimations for generalized hyperbolic (GH), Normal inverse Gaussian (NIG), i.e. GH with $\lambda = - \frac{1}{2}$,
hyperbolic (HYP), i.e. GH with $\lambda = 1$, bilateral Gamma, Variance Gamma (VG) and Normal distributions. In the following table
we see the Kolmogorov-distances ($L^{\infty}$), the $L^1$-distances and the $L^2$-distances between the empirical and the estimated distribution functions. The number in brackets denotes the number of parameters of the respective distribution family. Despite its practical relevance, we have omitted the class of CGMY distributions, because their probability densities are not available in closed form.

\begin{center}
\begin{tabular}{|l||c|c|c|}\hline
 & Kolmogorov-distance & $L^1$-distance & $L^2$-distance\\ \hline\hline
GH (5) & 0.0134 & 0.0003 & 0.0012\\ \hline
NIG (4) & 0.0161 & 0.0004 & 0.0013\\ \hline
HYP (4) & 0.0137 & 0.0004 & 0.0013\\ \hline
Bilateral (4) & 0.0160 & 0.0003 & 0.0013\\ \hline
VG (3) & 0.0497 & 0.0011 & 0.0044\\ \hline
Normal (2) & 0.0685 & 0.0021 & 0.0091\\ \hline
\end{tabular}
\end{center}  

We remark that the fit provided by bilateral Gamma distributions is of the same quality as that of NIG and HYP, the four-parameter subclasses of generalized hyperbolic distributions.

We perform the \textit{Kolmogorov test} by using the following table which shows the quantiles $\lambda_{1-\alpha}$ of order $1 - \alpha$ of the Kolmogorov distribution divided by the square root of the number $n$ of observations. Recall that in our example we have $n = 750$.

\begin{center}
\begin{tabular}{|c||c|c|c|c|c|}\hline
$\alpha$ & 0.20 & 0.10 & 0.05 & 0.02 & 0.01\\ \hline
$\lambda_{1-\alpha} / \sqrt{n}$ & 0.039 & 0.045 & 0.050 & 0.055 & 0.059\\ \hline
\end{tabular}
\end{center} 

Taking the Kolmogorov-distances from the previous table and comparing them with the values $\lambda_{1-\alpha} / \sqrt{n}$ of this box, we see that the hypothesis of a Normal distribution can clearly be
denied, Variance Gamma distribution can be denied with probability of error $5$ percent, whereas the remaining families of distributions cannot be rejected.

\begin{figure}[t]
   \centering
   \includegraphics[height=30ex,width=50ex]{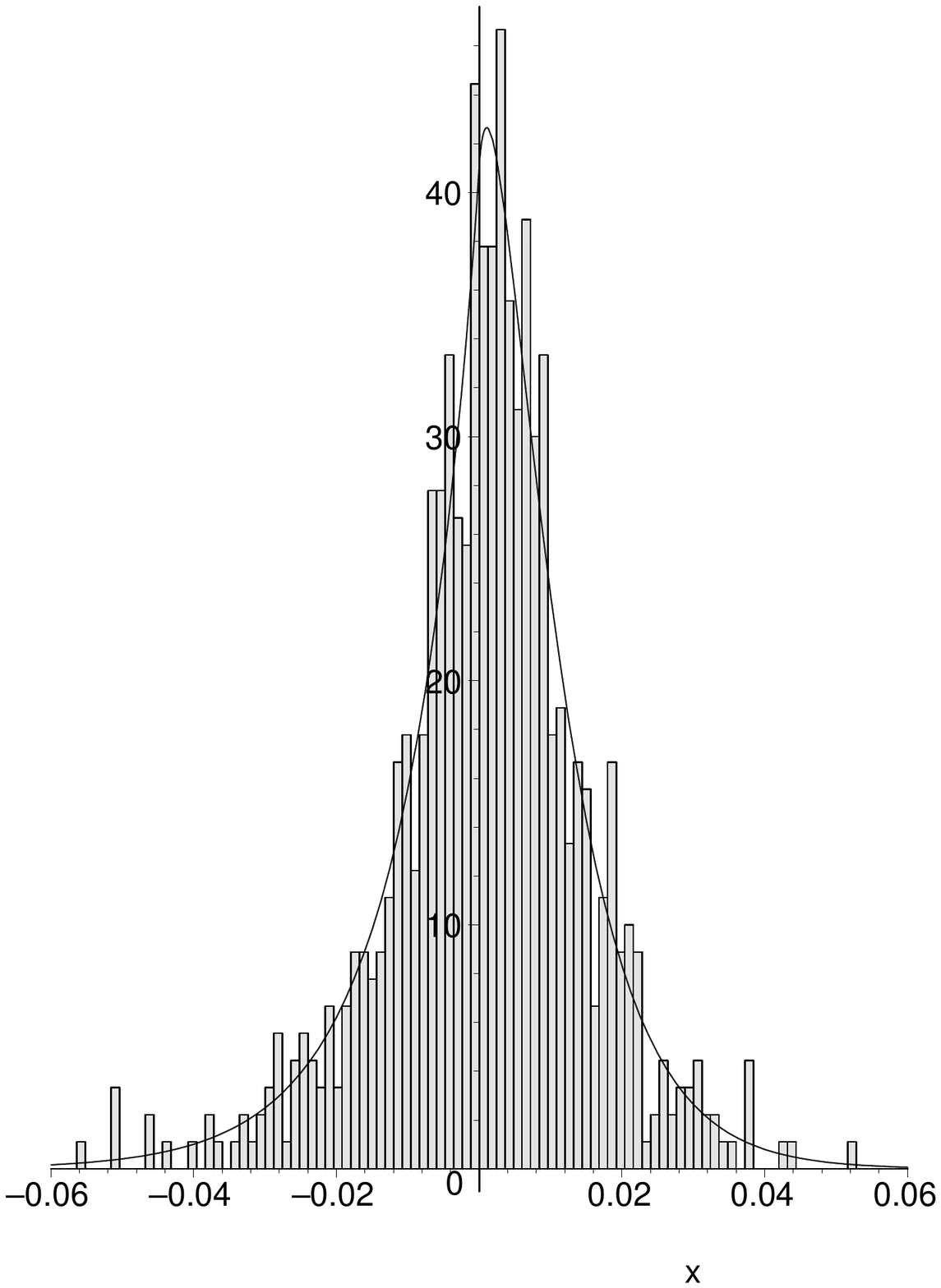}
   \caption{Empirical density and fitted bilateral Gamma density.}
   \label{fig-emp-bil-dens}
\end{figure}

\section{Term structure models}

Let $f(t,T)$ be a Heath-Jarrow-Morton term structure model (\cite{HJM})
\begin{align*}
df(t,T) = \alpha(t,T)dt + \sigma(t,T)dX_t,
\end{align*}
driven by a one-dimensional L${\rm \acute{e}}$vy process $X$. We
assume that the cumulant generating function $\Psi$ exists on some
non-void closed interval $I \subset \mathbb{R}$ having zero as inner point. By
equation (\ref{cumulant-gamma}), this condition is satisfied for
bilateral Gamma processes by taking any non-void closed interval $I \subset (-\lambda^-,\lambda^+)$ with zero as an inner point.

We assume that the volatility $\sigma$ is deterministic and that,
in order to avoid arbitrage, the drift $\alpha$ satisfies the HJM
drift condition
\begin{align*}
\alpha(t,T) = -\sigma(t,T) \Psi'(\Sigma(t,T)), \quad \text{where
$\Sigma(t,T) = - \int_t^T \sigma(t,s)ds$.}
\end{align*}
This condition on the drift is, for instance, derived in \cite[Sec. 2.1]{Eberlein_O}.
Since $\Psi$ is only defined on $I$, we impose the additional
condition
\begin{align}\label{canonical-condition}
\Sigma(t,T) \in I \quad \text{for all $0 \leq t \leq T$.}
\end{align}
It was shown in \cite{Eberlein-Raible} and \cite{K�chler-Naumann} that the short rate
process $r_t = f(t,t)$ is a \textit{Markov process} if and only if the
volatility factorizes, i.e. $\sigma(t,T) = \tau(t) \zeta(T)$. Moreover, provided differentiability of $\tau$ as well as $\tau(t) \neq 0$, $t \geq 0$ and $\zeta(T) \neq 0$, $T \geq 0$, there exists an affine \textit{one-dimensional realization}. Since $\sigma(\cdot,T)$ satisfies for each fixed $T \geq 0$ the ordinary differential equation
 \begin{align*}
\frac{\partial}{\partial t} \sigma(t,T) = \frac{\tau'(t)}{\tau(t)} \sigma(t,T), \quad t \in [0,T]
\end{align*}
we verify by using It\^o's formula \cite[Thm. I.4.57]{Jacod-Shiryaev} for fixed $T \geq 0$ that such a realization
\begin{align}\label{one-dim-real} 
f(t,T) = a(t,T) + b(t,T)Z_t, \quad 0 \leq t \leq T
\end{align}
is given by
\begin{align}\label{affine-comp-real}
a(t,T) = f(0,T) + \int_0^t \alpha(s,T)ds, \quad b(t,T) = \sigma(t,T)
\end{align} 
and the one-dimensional state process $Z$, which is the unique solution of the
stochastic differential equation
\begin{align*}
\left\{ \begin{array}{rcl} dZ_t & = & -\frac{\tau'(t)}{\tau(t)}
Z_t dt + dX_t \medskip
\\ Z_0 & = & 0 \end{array} \right. .
\end{align*}
We can transform this realization into an affine \textit{short rate realization}. By (\ref{one-dim-real}), it holds for the short rate $r_t = a(t,t) + b(t,t)Z_t$, $t \geq 0$, implying
\begin{align*}
Z_t = \frac{r_t - a(t,t)}{b(t,t)}, \quad t \geq 0.
\end{align*}
Inserting this equation into (\ref{one-dim-real}), we get
\begin{align*}
f(t,T) = a(t,T) + \frac{b(t,T)}{b(t,t)} ( r_t - a(t,t) ), \quad 0 \leq t \leq T.
\end{align*}
Incorporating (\ref{affine-comp-real}), we arrive at
\begin{align}\label{short-rate-real-general}
f(t,T) &= f(0,T) - \int_0^t \left[ \Psi'(\Sigma(s,T)) - \Psi'(\Sigma(s,t)) \right] \sigma(s,T) ds
+ \frac{\zeta(T)}{\zeta(t)} \left( r_t- f(0,t) \right).
\end{align}
As an example, let $f(t,T)$ be a term structure model having a \textit{Vasi\^cek} volatility structure, i.e.
\begin{align}
\sigma(t,T) = - \hat{\sigma} e^{-a(T-t)}, \quad 0 \leq t \leq T
\end{align}
with real constants $\hat{\sigma} > 0$ and $a \neq 0$. We assume that $a > 0$ and  $\frac{\hat{\sigma}}{a} < \lambda^+$. Since
\begin{align}\label{capital-sigma} 
\Sigma(t,T) = \frac{\hat{\sigma}}{a} \left( 1-e^{-a(T-t)} \right), \quad 0 \leq t \leq T
\end{align}
we find a suitable interval $I \subset (-\lambda^-,\lambda^+)$ such that
condition (\ref{canonical-condition}) is satisfied. By the results above,
the short rate $r$ is a Markov process and there exists a short rate realization. Equation (\ref{short-rate-real-general}) simplifies to
\begin{align}\label{short-rate-real-vas}
f(t,T) &= f(0,T) + \Psi(\Sigma(0,T)) - \Psi(\Sigma(t,T)) - e^{-a(T-t)} \Psi(\Sigma(0,t))
\\ \notag & \quad + e^{-a(T-t)} \left( r_t- f(0,t) \right).
\end{align} 
We can compute the bond prices $P(t,T)$ by using the following result.

\begin{proposition}
It holds for the bond prices
\begin{align*} 
P(t,T) = e^{\phi_1(t,T) - \phi_2(t,T)r_t}, \quad 0 \leq t \leq T
\end{align*}  
where the functions $\phi_1, \phi_2$ are given by 
\begin{align}\label{realization-part1}
\phi_1(t,T) &= -\int_t^T f(0,s)ds - \int_t^T \Psi \left( \frac{\hat{\sigma}}{a} \left( 1 - e^{-as} \right) \right) ds
\\ \notag &\quad + \int_t^T \Psi \left( \frac{\hat{\sigma}}{a} \left( 1 - e^{-a(s-t)} \right) \right) ds
\\ \notag &\quad + \frac{1}{a} \left( 1 - e^{-a(T-t)} \right) \left[ f(0,t) + \Psi 
\left( \frac{\hat{\sigma}}{a} \left( 1 - e^{-at} \right) \right) \right],
\\ \phi_2(t,T) &= \frac{1}{a} \left( 1 - e^{-a(T-t)} \right).
\end{align}
\end{proposition}

\begin{proof}
The claimed formula for the bond prices follows from the identity
\begin{align*} 
P(t,T) = e^{-\int_t^T f(t,s) ds}
\end{align*}
and equations (\ref{capital-sigma}), (\ref{short-rate-real-vas}).
\end{proof}

The problem is that $\phi_1$ in (\ref{realization-part1}) is difficult to compute for a general driving L\'evy process $X$, because we have to integrate over an expression involving the cumulant generating function $\Psi$. However, for bilateral Gamma processes we can derive (\ref{realization-part1}) in closed form. For this aim, we consider the \textit{dilogarithm
function} \cite[page 1004]{Abramowitz}, defined as
\begin{align*}
{\rm dilog}(x) := - \int_1^x \frac{\ln t}{t-1} dt, \quad x \in
\mathbb{R}_+
\end{align*} 
which will appear in our closed form representation. The dilogarithm function has the series expansion
\begin{align*}
{\rm dilog}(x) = \sum_{k=1}^{\infty} (-1)^k \frac{(x-1)^k}{k^2},
\quad 0 \leq x \leq 2
\end{align*}
and moreover the identity
\begin{align*}
{\rm dilog}(x) + {\rm dilog} \left( \frac{1}{x} \right) = -
\frac{1}{2} (\ln x)^2, \quad 0 \leq x \leq 1
\end{align*}
is valid, see \cite[page 1004]{Abramowitz}. For a computer program, the
dilogarithm function is thus as easy to evaluate as the natural
logarithm. The following auxiliary result will be useful for the computation of the bond prices $P(t,T)$.

\begin{lemma}\label{lemma-dilog}
Let $a,b,c,d,\lambda \in \mathbb{R}$ be such that $a \leq b$ and $c >
0 ,\lambda \neq 0$. Assume furthermore that $c+de^{\lambda x} > 0$
for all $x \in [a,b]$. Then we have
\begin{align*}
\int_a^b \ln \left( c + d e^{\lambda x} \right) dx = (b-a) \ln(c)
- \frac{1}{\lambda} {\rm dilog} \left( 1 + \frac{d}{c}
e^{\lambda b} \right) + \frac{1}{\lambda} {\rm dilog} \left( 1
+ \frac{d}{c} e^{\lambda a} \right).
\end{align*}
\end{lemma}

\begin{proof}
With $\varphi(x) := 1 + \frac{d}{c} e^{\lambda x}$ we obtain by
making a substitution
\begin{align*}
\int_a^b \ln \left( c + d e^{\lambda x} \right) dx &= (b-a) \ln(c)
+ \int_a^b \ln \left( 1 + \frac{d}{c} e^{\lambda x} \right) dx
\\ &= (b-a) \ln(c) + \frac{1}{\lambda} \int_{\varphi(a)}^{\varphi(b)} \frac{\ln t}{t-1} dt
\\ &= (b-a) \ln(c)
- \frac{1}{\lambda} {\rm dilog} \left( 1 + \frac{d}{c}
e^{\lambda b} \right) + \frac{1}{\lambda} {\rm dilog} \left( 1
+ \frac{d}{c} e^{\lambda a} \right).
\end{align*}
\end{proof}

Now assume the driving process $X$ is a bilateral Gamma process $\Gamma(\alpha^+, \lambda^+; \alpha^-, \lambda^-)$. We obtain a formula for the bond prices $P(t,T)$ in terms of the natural logarithm and the dilogarithm function.

\begin{proposition}
The function $\phi_1$ in (\ref{realization-part1}) has the representation
\begin{align*}
\phi_1(t,T) &= -\int_t^T f(0,s)ds 
\\ &+ \frac{\alpha^+}{a} [D_1(\lambda^+,T) - D_1(\lambda^+,t) 
- D_1(\lambda^+,T-t) + D_1(\lambda^+,0)]
\\ &+ \frac{\alpha^-}{a} [D_0(\lambda^-,T) - D_0(\lambda^-,t) 
- D_0(\lambda^-,T-t) + D_0(\lambda^-,0)]
\\ &+ \frac{1}{a} \left( 1 - e^{-a(T-t)} \right) [ f(0,t) + \alpha^+ L_1(\lambda^+) + \alpha^- L_0(\lambda^-) ],
\end{align*}
where
\begin{align*}
D_{\beta}(\lambda,t) &= {\rm dilog} \left( 1 + \frac{\hat{\sigma} e^{-at}}{\lambda^+ a + (-1)^{\beta} \hat{\sigma}}
\right), \quad \beta \in \{ 0,1 \},
\\ L_{\beta}(\lambda) &= \ln \left( \frac{\lambda a}{\lambda a + (-1)^{\beta} \hat{\sigma} (1 - e^{-at})} \right), \quad \beta \in \{ 0,1 \}.
\end{align*}
\end{proposition}

\begin{proof}
The assertion follows by inserting the cumulant generating function (\ref{cumulant-gamma}) of the bilateral Gamma process $X$ into (\ref{realization-part1}) and using Lemma \ref{lemma-dilog}.
\end{proof}

\section{Conclusion}

We have seen above that bilateral Gamma processes can be used for modelling financial data. 
One reason for that consists in their four parameters, which ensure good fitting properties. They share this number of parameters with several other classes of processes or distributions mentioned in Section \ref{sec-related}. 
Moreover, their trajectories have infinitely many jumps on every interval, which makes the models quite realistic. On the contrary to other well studied classes of L\'evy processes, these trajectories have finite variation on every bounded interval. Thus one can decompose every trajectory into its increasing and decreasing part and use it for statistical purposes. Other advantages of this class of processes are the simple form of the L\'evy characteristics and the cumulant generating function as well as its derivative. These enable a transparent construction of estimation procedures for the parameters and make the calculations in certain term structure models easy.\\

\end{document}